\documentclass[oneside]{amsart}

\usepackage{amsthm,amsmath,amssymb,amsfonts}
\usepackage{ytableau}
\usepackage[colorlinks=true]{hyperref}
\usepackage{cleveref}
\usepackage{geometry}
\usepackage{tikz}
\usepackage{array}
\usetikzlibrary{positioning}

\usepackage[msc-links, numeric]{amsrefs}
 
\usepackage{algorithm}
\usepackage{algorithmicx}
\usepackage{algpseudocode}

\usepackage{tikz-cd}
\usepackage{mathalpha}

\makeatletter
\newcommand{\cbigoplus}{\DOTSB\cbigoplus@\slimits@}
\newcommand{\cbigoplus@}{\mathop{\widehat{\bigoplus}}}
\makeatother

\title[Complex Representations of Groups and Involutions of its Automorphisms]{Complex Representations of Groups and Involutions of its Automorphisms}

\author[Yerrapati]{Venkata Subbaiah Yerrapati}
\address{(Yerrapati) Department of Mathematics, S. V. National Institute of Technology, Surat-7, Gujarat, India}
\email{yvsmath@gmail.com}
\thanks{The first author would like to thank the Department of Education, Government of India, for the financial assistance.}

\author[Dixit]{Rahul Dixit}
\address{(Dixit) Department of Artificial Intelligence, S. V. National Institute of Technology, Surat-7, Gujarat, India}
\email{rahuldixit@aid.svnit.ac.in}

\author[Shukla]{Ajay Kumar Shukla}
\address{(Shukla) Department of Mathematics, S. V. National Institute of Technology, Surat-7, Gujarat, India}
\email{aks@amhd.svnit.ac.in}

\theoremstyle{plain}
\newtheorem{theorem}{Theorem}[section] 
\newtheorem{lemma}[theorem]{Lemma} 
\newtheorem{proposition}[theorem]{Proposition}
\newtheorem{corollary}[theorem]{Corollary}

\theoremstyle{definition}
\newtheorem{definition}[theorem]{Definition} 

\theoremstyle{remark}
\newtheorem{remark}[theorem]{Remark} 

\newcommand{\ord}{ord}
\newcommand{\lcm}{lcm}

\begin{document}

 \begin{abstract}
       In this work, we establish a relationship between the sum of irreducible character degrees and the number of twisted involutions associated with the automorphisms of a finite group. We develop algorithmic frameworks for evaluating these quantities in the context of inner automorphisms and the symmetric group $\mathfrak{S}_n$. As an application, we provide a criterion for identifying groups that possess complex (non-real) irreducible representations and explore the structural consequences arising from these results.
    \end{abstract}

    \subjclass[2020]{20B25, 20B30, 20C30, 20D45}

    \keywords{automorphism, finite group representations, involutions, symmetric groups}
    
    \maketitle
    \section{Introduction}
    The representation theory of finite groups, particularly through the lens of Frobenius-Schur indicators, reveals deep connections between the internal structure of a group and its involutions. Originating from the work of Frobenius and Schur \cite{zbMATH02646565}, the concept of the indicator was introduced to connect irreducible representations with involutions. Specifically, a classical result states that for a finite group $G$, the number of solutions to the equation $g^2 = e$ coincides with the sum of the degrees of all distinct irreducible real representations of $G$. This foundational theory has been extended in several directions. Kawanaka and Matsuyama \cite{MR1078503} introduced the twisted Frobenius-Schur indicator for an automorphism $\sigma \in \mathrm{Aut}(G)$ satisfying $\sigma^2 = 1$. This notion was later generalized by Bump and Ginzburg \cite{MR2068079} to automorphisms $\sigma$ of arbitrary order $r$ (where $\sigma^r = 1$). More recently, a refined version of these indicators was proposed by Ceccherini-Silberstein, Scarabotti, and Tolli \cite{MR3320995}.
    
    Motivated by these developments, Des MacHale proposed a conjecture extending the work of Frobenius and Schur to other automorphisms and complex representations, which is recorded by Khukhro and Mazurov in the Kourovka Notebook \cite{MR4842130}*{Problem 16.60}. Some of the ideas were subsequently explored by Gow \cite{MR716800, MR808851} and Vinroot \cite{MR2173976}, who extracted significant information regarding the realizability of representations for finite linear groups. While the general theory applies to all finite groups, the representation theory of the symmetric group $\mathfrak{S}_n$ is of particular interest due to its deep intertwining with the combinatorics of integer partitions. A central invariant in this theory is the sum of the degrees of the irreducible complex representations, denoted by $T(\mathfrak{S}_n) = \sum_{\lambda \vdash n} f^\lambda$. This sequence, often denoted $a_n$, grows rapidly and satisfies the well-known recurrence $a_n = a_{n - 1} + (n - 1)a_{n - 2}$ for $n \geq 2$, where $a_0 = a_1 = 1$, a result first established by Chowla, Herstein, and Moore \cite{MR41849} and combinatorially interpreted through the removal of cells from Young tableaux \cite{MR1153249}. While the asymptotic behavior of $a_n$ is well-understood due to Moser and Wyman \cite{MR68564}, the interplay between this quantity, the automorphism group of $\mathfrak{S}_n$, and the fine structure of the partition lattice remains an active area of inquiry.

    Let $G$ be a finite group and let $\mathrm{Aut}(G)$ denote its automorphism group. For a fixed automorphism $\sigma \in \mathrm{Aut}(G)$, we consider the set
    \begin{equation}
        S_{\sigma} = \{g \in G \mid \sigma(g) = g^{-1}\}.
    \end{equation}
    
    The first focus of this paper is the extremal behavior of the set $S_\sigma$. The study of such twisted sets was formalized by Kawanaka \cite{MR1078564}, and Bump and Ginzburg \cite{MR2068079} through the twisted Frobenius-Schur indicator, which relates the cardinality of $S_\sigma$ to character values twisted by $\sigma$. A natural question arises: does ``twisting'' the group structure allow for a larger number of solutions to $\sigma(g) = g^{-1}$ than the standard condition $g^2 = 1$? We answer this in the negative. In Theorem \ref{half_main_result}, we prove that for a finite group, the quantity $T(G)$ serves as a sharp upper bound:
    \[ |S_\sigma| \le T(G). \]
    Furthermore, we provide a classification of the equality cases for symmetric groups, showing that if $\sigma$ is an inner automorphism induced by an element $x \in \mathfrak{S}_n$ satisfying $x^2=1$, then $|S_\sigma| = T(\mathfrak{S}_n)$. 
    
    Having established $T(\mathfrak{S}_n)$ as an extremal invariant, we turn to the fine structure of the sum itself. While the scalar recurrence $a_n$ is well-known, a closed-form summation that reveals the contribution of specific partition geometries is less transparent. By expanding the degree sum through the lens of the Hook Length Formula \cite{MR1824028}*{Theorem 3.10.2}, we derive a new decomposition of the involution count in Theorem \ref{sodoicrs}. We show that $a_n$ admits a representation as a sum of $\lfloor n/2 \rfloor$ distinct integer layers:
    \[ a_n = 1 + \sum_{k=0}^{\lfloor n/2 \rfloor - 1} A_k^{(n)}. \]
    
    This "layered" expansion suggests that the recurrence imposes a rigid filtration on the Young Lattice that is not captured by standard shape statistics. Based on computational evidence for $n \le 12$, we conjecture that these layers correspond to fibers of a map between the subset of partitions of $n$ and the set consisting of values $A_k^{(n)}$. 
    \section{Preliminaries}
    We recall some of the elementary results of symmetric groups.
    \subsection{Symmetric groups}    \begin{remark}\label{number_of_second_order_terms}
        Let $\mathfrak{S}_n$ denote the symmetric group of $n$ elements. The number of elements of order $2$ is given by
        \begin{enumerate}
            \item $\Sigma_{k = 0}^{\frac{n}{2} - 1} \frac{n!}{2^{\frac{n - 2k}{2}}\left(\frac{n - 2k}{2}\right)!\left(2k\right)!} $, if $n$ is even.
            \item $\Sigma_{k = 0}^{\frac{n - 1}{2} - 1} \frac{n!}{2^{\frac{n - (2k + 1)}{2}}\left(\frac{n - (2k + 1)}{2}\right)!\left(2k + 1\right)!}$, if $n$ is odd.
        \end{enumerate}
    \end{remark}
    \begin{remark}\label{s_n_generators}
        The symmetric group $\mathfrak{S}_n$ is generated by the transpositions $(1,2), (2,3), \ldots, (n - 1, n)$.
    \end{remark}
    \begin{remark}\label{homomorphism_preserves_conjugacy_class}
        Let $\varphi : \mathfrak{S}_n \to \mathfrak{S}_n$ be a group homomorphism. Suppose that there exists $a \in [x]$ such that $\varphi(a) \in [y]$, then $\varphi(b) \in [y]$ for all $b \in [x]$, where $[x]$ is the conjugacy class of $x$ in $\mathfrak{S}_n$.
    \end{remark}
    \begin{remark}\label{automorphism_properties_of_S_n}
     Let $\varphi : \mathfrak{S}_n\to\mathfrak{S}_n$ be an automorphism. The following properties hold.
        \begin{enumerate}
            \item  Let $y$ be an element of $\mathfrak{S}_n$ then $\ord (y) = \ord (\varphi(y))$.
            \item Let $x$ be an element of $\mathfrak{S}_n$ and $[\varphi(x)]$ the conjugacy class $\varphi$ 
            \begin{enumerate}
                \item $\big|[x]\big| = \big|[\varphi(x)]\big|$.
                \item Suppose $x$ is 2-cycle and $[x] \cap [\varphi(x)] \neq \emptyset$ then $\varphi(x)$ is also a 2-cycle.
            \end{enumerate}
        \end{enumerate}
    \end{remark}
    The proof to the following Proposition can be found in \cite{MR1307623}*{Lemma 7.4}.
    \begin{proposition}\label{inner_if_transpositions}
        Let $\varphi : \mathfrak{S}_n\to\mathfrak{S}_n$ be an automorphism then $\varphi$ is inner automorphism if and only if it preserves transpositions.
    \end{proposition}
    \begin{definition}
        A group $G$ is \emph{complete} if it is center less and every automorphism of $G$ is inner.
    \end{definition}
    \begin{lemma}[\cite{MR1307623}*{Theorem 7.5}]\label{all_complete}
        Let $n$ be a natural number with $n \neq 2, 6$ then the symmetric group $\mathfrak{S}_n$ is complete.
    \end{lemma}
    \begin{lemma}\label{all_automorphisms}
        Let $n$ be a natural number with $n \neq 6$ and $n \geq 3$. Suppose $\varphi : \mathfrak{S}_n\to\mathfrak{S}_n$ is an automorphism then there exists $x \in \mathfrak{S}_n$ such that $\varphi(y) = xyx^{-1}$ for all $y \in \mathfrak{S}_n$.
    \end{lemma}
    \begin{proof}
        This follows from Proposition \ref{inner_if_transpositions} and Lemma \ref{all_complete}.
    \end{proof}
    \begin{lemma}\cite{MR1307623}*{Theorem 3}\label{S_6_holder}
        There exists an outer automorphism of $\mathfrak{S}_6$.
    \end{lemma}
    \begin{definition}
        A \emph{syntheme} is a product of 3 disjoint transpositions. A \emph{pentad} is a family of 5 synthemes, no two of which have a common transpositions.
    \end{definition}
    \begin{lemma}\cite{MR1307623}*{Lemma 7.11}\label{basis_for_outer_automorphisms_of_S6}
        $\mathfrak{S}_6$ contains exactly 6 pentads. They are
        \begin{align*}
            (12)(34)(56), (13)(25)(46), (14)(26)(35), (15)(24)(36), (16)(23)(45)\\
            (12)(34)(56), (13)(26)(45), (14)(25)(36), (15)(23)(46), (16)(24)(35)\\
            (12)(35)(46), (13)(24)(56), (14)(25)(36), (15)(26)(34), (16)(23)(45)\\
            (12)(35)(46), (13)(26)(45), (14)(23)(56), (15)(24)(36), (16)(25)(34)\\
            (12)(36)(45), (13)(24)(56), (14)(26)(35), (15)(23)(46), (16)(25)(34)\\
            (12)(36)(45), (13)(25)(46), (14)(23)(56), (15)(26)(34), (16)(24)(35).
        \end{align*}
    \end{lemma}
    \begin{lemma}\cite{MR1307623}*{Theorem 7.12}\label{all_outer_automorphisms_forms_of_S6}
        If $\{\sigma_2, \ldots, \sigma_6\}$ is a pentad in some ordering then there is unique outer automorphism $\gamma$ of $\mathfrak{S}_6$ with $\gamma : (1, i) \mapsto \sigma_i$ for $i = 2, 3, 4, 5, 6$. Moreover, every outer automorphism of $\mathfrak{S}_6$ has this form.
    \end{lemma}
    \begin{definition}
        Let $\sigma$ be an automorphism of $\mathfrak{S}_n$, an element $\pi$ of $\mathfrak{S}_n$ is called an \emph{involution} of $\sigma$ if $\sigma(\pi) = \pi^{-1}$.
    \end{definition}
    \begin{lemma}\label{identity_automorphism_involutions}
        Let $n$ be a natural number then the number of elements of order two in $\mathfrak{S}_n$ plus one is equal to the number of involutions of identity automorphism of $\mathfrak{S}_n$.
    \end{lemma}
    \subsection{Young Diagrams}
    \begin{definition}[Partition]
        Let $n$ be a natural number, a \emph{partition} $\lambda = (\lambda_1, \lambda_2, \cdots, \lambda_n )$ of $n$ is the sequence of non-increasing order of positive integers $\lambda_1, \lambda_2, \cdots, \lambda_n$ such that $\lambda_1 + \lambda_2+ \cdots + \lambda_n = n$.
    \end{definition}
    
    For instance, let $n = 3$ then $\lambda = (1, 1, 1),\ \mu = (2, 1)$ are partitions of $3$. The \emph{young diagram} is a diagram consisting of boxes which are drawn based on the partition of a natural number $n$ and the partition $\lambda$ of $n$ is called the \emph{shape} of the young diagram.
   
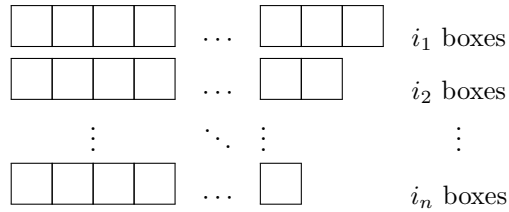
\begin{align*}
       \begin{array}{c c l c}
        \ydiagram{4} & \cdots &\ydiagram{3} & i_1 \text{ boxes}\\
        \ydiagram{4} & \cdots &\ydiagram{2} & i_2 \text{ boxes}\\
        \vdots & \ddots & \vdots & \vdots\\
        \ydiagram{4} & \cdots & \ydiagram{1}& i_n \text{ boxes}
        \end{array}
   \end{align*}
    For example, one of the partition of 21 is $i_1 = 6,\ i_2 = 5,\ i_3 = 4,\ i_4 = 3,\ i_5 = 2,\ i_6 = 1,\ i_7 = i_8 = \cdots = i_{21} = 0$ and the corresponding young diagram of the partition is
    \begin{center}
        \ydiagram{6, 5, 4, 3, 2, 1}
    \end{center}
    Assuming that the boxes of the young diagram can be indexed as $(1, 1), (1,2), (1,3), \ldots, (2, 1), (2,2) \ldots$. The definition of the hook length at index $(i, j)$ is as follows.
    
    \begin{definition}[Hook Length]
        The \emph{hook length} of a young diagram at the index $(i, j)$ is the summation of the number of entries (boxes) to the right of the box of $(i, j)^{th}$ index and the number of entries to the bottom of the box of $(i, j)^{th}$ index plus one. The hook length at the index is denoted as $h_{i,j}$
        .
    \end{definition}
    Now, we turn to the dimension formula that counts the number of Standard Young tableaux of a given shape $\lambda$ of size $n$. This count determines the dimension of the corresponding irreducible representation of the symmetric group and given by $\frac{n!}{\Pi_{i, j}h_{i, j}}$ \cite{MR1824028}*{Theorem 3.10.2}, where $h_{i,j}$ denote the hook length of the $(i,j)^{th}$ box of the young diagram. To illustrate this, let us consider the symmetric group $\mathfrak{S}_3$. The number $3$ has three possible partitions and their young diagrams are
    \ydiagram{3}, \ydiagram{2, 1}, \ydiagram{1, 1, 1}.

    The dimensions of corresponding irreducible representation to partitions of $3$ are 
    \begin{enumerate}
        \item \begin{center}
            \ydiagram{3} 
        \end{center}
    $h_{1,1} = 2 + 0 + 1 = 3, h_{1, 2} = 1 + 0 + 1 = 2, h_{1, 3} = 0 + 0 + 1 = 1$ and dimension is  $\frac{3!}{h_{1, 1}\cdot h_{1,2}\cdot h_{1, 3}} = \frac{3!}{3\cdot 2 \cdot 1} = 1$
    \item \begin{center}
        \ydiagram{2, 1} 
    \end{center}
    $h_{1,1} = 1 + 1 + 1 = 3, h_{1, 2} = 0 + 0 + 1 = 1, h_{2, 1} = 0 + 0 + 1 = 1$ and dimension is $\frac{3!}{h_{1, 1}\cdot h_{1,2}\cdot h_{2, 1}} = \frac{3!}{3\cdot 1 \cdot 1} = 2$
    \item \begin{center}
        \ydiagram{1, 1, 1} 
    \end{center}
    $h_{1,1} = 0 + 2 + 1 = 3, h_{2, 1} = 0 + 1 + 1 = 2, h_{2, 1} = 0 + 0 + 1 = 1$ and dimension is $\frac{3!}{h_{1, 1}\cdot h_{2,1}\cdot h_{3, 1}} = \frac{3!}{3\cdot 2 \cdot 1} = 1$
    \end{enumerate}

     The sum of degrees of all irreducible representations of symmetric group, which can derived from corresponding young diagrams is equal to sum of dimensions of each young tableaux i.e., $1 + 2 + 1 = 4$ and these representations can be explicitly constructed by using Specht modules (for detailed study, see \cite{MR1824028}*{Chapter 2}). 

\subsection{Robinson-Schensted Algorithm}
    Let $\pi \in \mathfrak{S}_n$, we construct a tableaux pairs.
    \begin{align*}
        (P_0, Q_0) = (\Phi, \Phi), (P_1, Q_1), \ldots , (P_n, Q_n) = (P, Q)
    \end{align*}
    where $x_1, x_2, \cdots x_n$ are inserted into the $P$'s and $1, 2, \cdots n$ are arranged in the $Q$'s and $P$ is called as \textit{insertion tableau} and $Q$ called as \textit{recording tableau}. Arranging these elements is described by the following algorithm (we confine to look at row insertion). Let us insert $x \in \{1, 2, \ldots , n\}$.
    
    \begin{algorithm}
    \caption{Robinson-Schensted Algorithm} \label{alg:cap}
    \begin{algorithmic}[1]
        \Require $x$
        \State Set Row $R$ := the first row of $P$;
        \While {$x <$ some element of row $R$}
        \State Let $y$ be smallest element of $R$ greater than $x$;
        \State Replace $y$ by $x$ in $R$;
        \State Set $x:=y$ and $R:=$ the next row down;
        \EndWhile
        \State Arrange $x$ at the end of row $R$;
    \end{algorithmic}
    \end{algorithm}
    
    \begin{theorem}[Symmetry Theorem, \cite{MR1464693}*{pp. 40}] \label{symmetry_theorem}
        If $\pi \in \mathfrak{S}_n$ corresponds to $(P, Q)$, then $\pi^{-1} \in \mathfrak{S}_n$ corresponds to $(Q, P)$.
    \end{theorem}
    
    \begin{lemma}[\cite{MR1464693}*{pp. 40}]\label{frobenius_Schur}
        The sum of degrees of irreducible complex representations of $\mathfrak{S}_n$ is equal to the number of involution's of an identity map.
    \end{lemma}
    \begin{proof}
        For an involution $\pi \in \mathfrak{S}_n$, $\pi = \pi^{-1}$. By Theorem \ref{symmetry_theorem}, if $(P, Q)$ corresponding sequence of tableau of $\pi$, then $P = Q$. This implies that involutions in the symmetric group $\mathfrak{S}_n$ correspond to pairs $(P,P)$ with $P$ as standard tableau with $n$ boxes; so there is a one-to-one correspondence between involutions and standard tableaux and by Lemma \ref{identity_automorphism_involutions}, we complete the proof.
    \end{proof}
    
    As a direct consequence of Schur's Lemma and the standard orthogonality relations (\cite{MR1824028}*{Chapter 1}), we obtain the following lemma.

    \begin{lemma}\label{bound_of_tfsi}
        Let $G$ be a finite group and let $\chi$ be an irreducible character of $G$. For any automorphism the twisted Frobenius-Schur indicator $\iota_{\alpha}(\chi) \leq 1$.
    \end{lemma}
    \begin{proof}
        \begin{equation}
            \iota_{\alpha}(\chi) = \frac{1}{|G|} \sum_{g \in G} \chi(g \alpha(g))
        \end{equation}
        On applying Cauchy-Schwarz inequality, we obtain
        \begin{equation}
            \left|\iota_{\alpha}(\chi)\right|\leq \left(\frac{1}{|G|} \sum_{g \in G} \left|\chi(g)\right|^2\right)^{\frac{1}{2}}\left(\frac{1}{|G|} \sum_{g \in G} \left|\chi(\alpha(g))\right|^2\right)^{\frac{1}{2}}
        \end{equation}
        Due to orthogonal relations, we have $\frac{1}{|G|} \sum_{g \in G} \left|\chi(g)\right|^2 \leq 1, \frac{1}{|G|} \sum_{g \in G} \left|\chi(\alpha(g))\right|^2 \leq 1$. Therefore $\left|\iota_{\alpha}(\chi)\right| \leq 1$.
    \end{proof}
    \begin{lemma}{\label{involution_set_relation}}
        Let $G$ be a finite group and $\alpha \in \mathrm{Aut}(G)$. Suppose $S_{\alpha}$ is the set of involutions corresponding the automorphism $\alpha$. Then $| S_{\alpha}| = \sum_{\chi \in \mathrm{Irr}(G)}\iota_{\alpha}(\chi) \chi(1)$
    \end{lemma}
    \begin{proof}
         For $\chi \in \mathrm{Irr}(G)$ and $\alpha \in \text{Aut}(G)$, the twisted Frobenius-Schur indicator is given by
        \begin{equation}
            \iota_{\alpha}(\chi) = \frac{1}{|G|} \sum_{g \in G} \chi(g \alpha(g))
        \end{equation}
        We compute the weighted sum of these indicators by the degrees $\chi(1)$:
        \begin{align*}
            \sum_{\chi \in \mathrm{Irr}(G)} \iota_{\alpha}(\chi) \chi(1) &= \sum_{\chi \in \mathrm{Irr}(G)} \left( \frac{1}{|G|} \sum_{g \in G} \chi(g \alpha(g)) \right) \chi(1) \\
            &= \frac{1}{|G|} \sum_{g \in G} \sum_{\chi \in \mathrm{Irr}(G)} \chi(1) \chi(g \alpha(g))
        \end{align*}
        Recall the property of the regular character $\rho_{\text{reg}} = \sum \chi(1)\chi$. We know that $\rho_{\text{reg}}(h) = |G|$ if $h=1$ and $0$ otherwise. Here, the argument is $h = g \alpha(g)$.
        \begin{itemize}
            \item If $g \alpha(g) = 1$, then $\alpha(g) = g^{-1}$, implying $g \in S_{\alpha}$.
            \item If $g \alpha(g) \neq 1$, the inner sum is 0.
        \end{itemize}
        Thus, the sum simplifies to counting the elements in $S_{\alpha}$:
        \begin{equation}
            \sum_{\chi \in \mathrm{Irr}(G)} \iota_{\alpha}(\chi) \chi(1) = \frac{1}{|G|} \sum_{g \in S_{\alpha}} |G| = |S_{\alpha}|
        \end{equation}
    \end{proof}
    \section{Main Results}

    \begin{theorem}\label{half_main_result}
         Let the sum of the degrees of all the irreducible complex representations of a finite group $G$ is $T(G) = \Sigma_{\chi \in \mathrm{Irr}(G)}\chi(1)$, where $\mathrm{Irr}(G)$ is the set of irreducible complex representations of  $G$. Let $\alpha\in Aut \ G $, consider the set $S_{\alpha} = \{g \in G \ | \ \alpha(g) = g^{-1}\}$ which corresponds to $\alpha$, then $T(G) \geq |S_{\alpha}|$.
    \end{theorem}
    \begin{proof}
        From Lemma \ref{involution_set_relation}, the number of twisted involutions for the automorphism $\alpha$ of finite group $G$ is given by the character sum $|S_{\alpha}| = \sum_{\chi \in \mathrm{Irr}(G)} \iota_{\alpha}(\chi) \chi(1)$.
        On applying Lemma \ref{bound_of_tfsi}, the twisted Frobenius-Schur indicator satisfies $|\iota_{\alpha}(\chi)| \leq 1$, for all $\chi \in \mathrm{Irr}(G)$. Applying this bound and the triangle inequality, we obtain
        \begin{equation}
        |S_{\alpha}| = \left| \sum_{\chi \in \mathrm{Irr}(G)} \iota_{\alpha}(\chi) \chi(1) \right| \leq \sum_{\chi \in \mathrm{Irr}(G)} |\iota_{\alpha}(\chi)| \chi(1) \leq \sum_{\chi \in \mathrm{Irr}(G)} 1 \cdot \chi(1) = T(G).
        \end{equation}
        Therefore, $|S_{\alpha}| \leq T(G)$.
    \end{proof}

    As a significant consequence, Theorem \ref{half_main_result} provides a criterion for determining whether a group admits purely complex representations, leading to the following corollary.
    \begin{corollary}
        Let $G$ be a finite group and let $\alpha$ an automorphism of $G$ such that $|S_{\alpha}| > |S_{id}|$. There is an irreducible complex representation of $G$.
    \end{corollary}
    \begin{theorem}\label{final_identity_greater_other}
        Let $\sigma$ be an automorphism of $\mathfrak{S}_n$. The number of involutions of identity automorphism is greater than or equal to the number of involutions of $\sigma$.
    \end{theorem}
    While Lemma \ref{frobenius_Schur} shows that the number of involutions of identity automorphisms of the symmetric group $\mathfrak{S}_n$ is equal to the sum of the degrees of all irreducible representations of the symmetric group $\mathfrak{S}_n$.  Although Theorem \ref{final_identity_greater_other} follows theoretically from Lemma \ref{frobenius_Schur} and Theorem \ref{half_main_result}, our intention is to provide a constructive proof by defining an explicit injective map from the set of involutions of an arbitrary automorphism to set of involutions of identity automorphism of a symmetric group. However, finding such a map for every case presents significant combinatorial challenges, we limit our explicit construction to specific subcases in this work. While a uniform construction for all cases remains a subject for future investigation. To define such a map, it is essential to first understand the structure of the involutions. The following lemma serves as the first step in this direction.
    \begin{lemma}\label{all_involutions_of_greater_order_3}
            Let $n$ be a natural number with $n \neq 6$, $x\in\mathfrak{S}_n$ such that $\ord (x) \geq 3$ and $\ord(x)$ is odd. Suppose that the automorphism corresponding to $x$ is $\sigma : \mathfrak{S}_n \rightarrow \mathfrak{S}_n$, defined by $\sigma(y) = xyx^{-1}$ for all $y \in \mathfrak{S}_n$. If $\sigma(y) = y^{-1}$, then $\ord (y) = 1$ or $\ord (y) = 2$.
        \end{lemma}
        
        \begin{proof}
            Suppose, on the contrary, there exist $y \in \mathfrak{S}_n$ with $\ord (y) = k \geq 3$ such that $\sigma(y) = y^{-1}$, then we explore the following possibilities.
            \begin{itemize}
                \item[Case A.] If $y$ is a product of m-disjoint cycles of same order $k$.
                \begin{align*}
                    y = (y_1, y_2, y_3, \ldots , y_k)(y_{k + 1}, y_{k + 2}, \ldots, y_{2k})\cdots(y_{(m - 1)k + 1}, y_{(m - 1)k + 2}, \ldots ,y_{mk})
                \end{align*}
                The following are the possible cases of $y$ such that $y$ maps to $y^{-1}$ under $\sigma$. 
                \begin{itemize}
                    \item[Case (1)] Each cycle of $y$ is mapping to its own inverse.
                    \begin{align*}
                        y = (y_1, y_2, \ldots, y_k)\gamma
                    \end{align*}
                    where $\gamma$ is the product of remaining cycles that are appearing in $y$. Since $(y_1, y_2, y_3, \ldots , y_k)$ maps to its inverse $(y_k, \ldots y_2, y_1) = y^{-1}$ under $\sigma$. Assuming that cycle is already rearranged in a way that implies
                    \begin{align*}
                        x = \begin{cases}
                            (y_1, y_k)(y_2, y_{k - 1})\cdots(y_{\frac{k}{2}}, y_{\frac{k}{2} + 1})\delta_1 & \text{if } k \text{ is even}\\
                            (y_1, y_k)(y_2, y_{k - 1})\cdots(y_{\frac{k - 1}{2}}, y_{\frac{k - 1}{2} + 2})\delta_2 & \text{ if } k \text{ is odd}
                        \end{cases}
                    \end{align*}
                    where $\delta_1, \delta_2$ are the product of disjoint transpositions which sends the cycles of $\gamma$ to their own inverses in the respective cases. So
                    \begin{align*}
                        \ord(x) &= \lcm\{2, \ord(\delta_1)\} \text{ or } lcm\{2, \ord(\delta_2)\}\\
                        &= 2j_1 \text{ for some $j_1 \in \mathbb{N}$}
                    \end{align*}
                    which contradicts the fact that $\ord(x)$ is odd.
                    \item[Case (2).] Each cycle of $y$ is mapping to other cycle's inverse.
                    \begin{align*}
                        y = (y_1, y_2, y_3, \ldots, y_k)(y_{k + 1}, y_{k + 2}, \ldots, y_{2k}) \ldots (y_{(m - 1)k + 1}, y_{(m - 1)k + 2}, \ldots, y_{mk})
                    \end{align*}
                    We prove it by induction on number of disjoint cycles. Suppose $(y_1, y_2, y_3, \ldots , y_k) \mapsto (y_{(s- 1)k + 1}, y_{(s- 1)k + 2}, \ldots , y_{sk})$ and $(y_{(s- 1)k + 1}, y_{(s- 1)k + 2}, \ldots , y_{sk}) \mapsto (y_1, y_2, y_3, \ldots , y_k)$. Assuming that cycle is already rearranged in a way, which implies that $x$ is of the form
                    \begin{align*}
                        x = (y_1, y_{sk})(y_2, y_{sk - 1}) \cdots (y_k, y_{(s - 1)k + 1})\ \delta_3
                    \end{align*}
                    Where $\delta_3$ is product of other disjoint cycles and the order of $x$ is \textit{lcm}$\{2, \ord(\delta_1)\}$  which is equal to $2j_2$ for some $j_2 \in \mathbb{N}$, which contradicts the fact that $\ord(x)$ is odd.
    
                    When $3$-cycles are mapping to their inverses cyclically.
                    \begin{align*}
                        (y_1, y_2, y_3, \ldots , y_k) &\mapsto (y_{(s- 1)k + 1}, y_{(s- 1)k + 2}, \ldots , y_{sk})^{-1}\\
                        (y_{(s- 1)k + 1}, y_{(s- 1)k + 2}, \ldots , y_{sk})&\mapsto (y_{(r- 1)k + 1}, y_{(r- 1)k + 2}, \ldots , y_{rk})^{-1}\\
                        (y_{(r- 1)k + 1}, y_{(r- 1)k + 2}, \ldots , y_{rk})&\mapsto (y_1, y_2, y_3, \ldots , y_k)^{-1}
                    \end{align*} Assuming that cycle is already rearranged in a way, which implies that $x$ is of the form
                    \begin{align*}
                        x = (y_1, y_{sk}, y_{(r - 1)k + 1}, y_k, y_{(s - 1)k + 1}y_{rk})\delta_4
                    \end{align*}
                    where $\delta_4$ is product of other disjoint cycles and the order of $x$ \textit{lcm} $\{6, \ord(\delta_4)\}$ which is equal to $2j_3$ for some $j_3 \in \mathbb{N}$ which contradicts the fact that $\ord(x)$ is odd. Now we assume it is true for $m_1$-disjoint cycles where each cycle is mapping to other cyclically.
                    \begin{align*}
                        (y_1, y_2, y_3, \ldots , y_k) &\mapsto (y_{(s- 1)k + 1}, y_{(s- 1)k + 2}, \ldots , y_{sk})^{-1}\\
                        (y_{(s- 1)k + 1}, y_{(s- 1)k + 2}, \ldots , y_{sk})&\mapsto (y_{(r- 1)k + 1}, y_{(r- 1)k + 2}, \ldots , y_{rk})^{-1}\\
                        (y_{(r- 1)k + 1}, y_{(r- 1)k + 2}, \ldots , y_{rk})&\mapsto \cdots\\
                        (y_{(m_1- 1)k + 1}, y_{(m_1- 1)k + 2} \ldots y_{m_1k})&\mapsto (y_1, y_2, y_3, \ldots , y_k)^{-1}
                    \end{align*}
                    \begin{itemize}
                        \item[Case (2.a)] If $m_1$ is odd, then on assuming that the cycle is already rearranged in a way which implies that $x$ is of the form
                        \begin{align*}
                            x = (y_1, y_{2k}, y_{2k + 1}, \ldots , y_{(m_1 - 1)k + 1}, y_k, y_{k + 1}, \ldots, y_{(m_1 - 2)k + 1, y_{m_1k}})\delta_5
                        \end{align*} where the cycle is of order $2m_1$ and $\delta_5$ is disjoint product of other cycles
                        \item[Case (2.b)] If $m_1$ is even, then on assuming that the cycle is already rearranged in a way which implies that $x$ is of the form
                        \begin{align*}
                            x = (y_1, y_{2k}, y_{2k + 1}, \ldots , y_{m_1k})\delta_6
                        \end{align*} where the cycle is of order $m_1$ and $\delta_6$ is disjoint product of other cycles
                    \end{itemize}
                    Now we prove for $m_1 + 1$ disjoint cycles where each cycle is mapping to other cyclically.
                    \begin{align*}
                        &(y_1, y_2, y_3, \ldots , y_k) \mapsto (y_{(s- 1)k + 1}, y_{(s- 1)k + 2}, \ldots , y_{sk})\mapsto \\
                        &(y_{(r- 1)k + 1}, y_{(r- 1)k + 2}, \ldots , y_{rk})\mapsto \ldots
                        (y_{(m_1- 1)k + 1}, y_{(m_1- 1)k + 2} \ldots y_{m_1k})\mapsto\\
                        & (y_{(m_1)k + 1}, y_{m_1k + 2} \ldots y_{(m_1 + 1)k})\mapsto (y_1, y_2, y_3, \ldots , y_k)
                    \end{align*}
                    \begin{itemize}
                        \item[Case (2.b.i)] If $m_1$ is odd i.e., $m_1 + 1$ is even, then we have $x$ of the form
                        \begin{align*}
                            (y_1, y_{2k}, \ldots y_{(m_1 - 1)k + 1}, y_{(m_1 + 1)k})\delta_7
                        \end{align*}where the cycle is of order $m_1 + 1$ and $\delta_6$ is disjoint product of other cycles ($m_1 + 1 = 2j_4$ for some $j_4 \in \mathbb{N}$.) and the order of $x$ is \textit{lcm}$\{2j_4, \ord(\delta_7)\}$ which is equal to $2j_5$ for some $j_5 \in \mathbb{N}$ which contradicts the fact that $\ord(x)$ is odd.
                        
                        \item[Case (2.b.ii)] If $m_1$ is even i.e., $m_1 + 1$ is odd, then we have $x$ of the form
                        \begin{align*}
                            (y_1, y_{2k}, \ldots y_{m_1k + 1}, y_{(m_1 + 1)k}, y_k, y_{k + 1}, \ldots y_{(m_1 - 1)k + 1}, y_{(m_1 + 1)k})\delta_8
                        \end{align*} where the cycle is of order $2(m_1 + 1)$ and $\delta_8$ is disjoint product of other cycles and the order of $x$ is \textit{lcm}$\{2(m_1 + 1), \ord(\delta_8)\}$ which is equal to $2j_6$ for some $j_6 \in \mathbb{N}$ which contradicts the fact that $\ord(x)$ is odd.
                    \end{itemize}
                    Hence by Mathematical Induction this is true for all natural numbers greater than or equal to 3.
                    \item Some of them mapped to their inverse and others are mapped to other's inverse. This case follows either from first part or second part of this case.
                \end{itemize}
                \item[Case B.] If $y$ is product of different order disjoint cycles. 
                Since $y$ contains $(y_1, y_2, \ldots y_k)$ and no other cycle of same order. So, $(y_1, y_2, \ldots y_k)$ should map to $(y_k, y_{k - 1}, \ldots y_1)$ under $\sigma$
                    \begin{enumerate}
                        \item[Case (1).] If $k$ is even. \begin{align*}
                            y = (y_1, y_2, \ldots y_k) \rightarrow (y_k, \ldots y_2, y_1) = y^{-1}
                        \end{align*}
                        Which means $y_1$ goes to $y_k$ and $y_k$ goes to $y_1$. A similar arrangement is assigned to other elements. This shows that $x$ is the product of $k/2$ disjoint transpositions, which implies $\ord (x)$ is even, which is a contradiction to the fact that $\ord(x)$ is odd.
                        \item[Case (2).] If $k$ is odd.
                        \begin{align*}
                            y = (y_1, y_2, \ldots y_k) \rightarrow (y_k, \ldots y_2, y_1) = y^{-1}
                        \end{align*}
                        Which means $y_1$ goes to $y_k$ and $y_k$ goes to $y_1$. A similar arrangement is assigned to other elements. This shows that $x$ is the product of $(k-1)/2$ disjoint transpositions, which implies $\ord(x)$ is even order, which is again a contradiction to the fact that $\ord(x)$ is odd.
                    \end{enumerate}
                \item[Case C.] If $y$ is a product of disjoint cycles such that some of them have the same order and some have different orders.

                This is followed by similar arguments of Case (B).
            \end{itemize}
            
            This completes the proof.
        \end{proof}
        Due to Lemma \ref{all_involutions_of_greater_order_3}, our attention is reduced to the inner automorphisms corresponding to elements $ x \in \mathfrak{S}_n$ such that $\ord(x)$ is even. To proceed, we propose an algorithm approach that enables to achieve the injective map $f: S_{\sigma} \to S_I$ for the case when the order of \( x \in \mathfrak{S}_n \) is a multiple of 4. As this category is broad, our algorithm targets permutations with a particular ordering structure. that enables us to achieve the required injective map.
        
        Before presenting the algorithm, that defines the injective map $f: S_{\sigma} \to S_{id}$. We introduce the necessary preliminary definitions and explicitly construct $f$ through the algorithm for this special case, while the remaining cases are addressed using well-defined maps, which will be introduced in Theorem~\ref{identity_greater_other}.
    \subsection{Algorithm}\label{algorithm}
    
        \subsubsection{Ordering elements}\label{ordering}
        Given $y \in \mathfrak{S}_n$, $y$ can be rewritten as        \begin{align*}\label{y_general_form}
            y = &(a_1, a_2, a_3, \ldots, a_{k_1})(a_{k_1 + 1}, a_{k_1 + 2}, \ldots , a_{k_1 + k_2}) \cdots \\
            &(a_{k_1 + k_2+ \cdots + k_{m - 1} + 1}, a_{k_1 + k_2 + \cdots + k_{m - 1} + 2}, \ldots , a_{k_1 + k_2 + \cdots + k_m})
        \end{align*}
        where $k_i \in \mathbb{N}$ for $i = 1, 2, \ldots, m$ and $a_j \in \{1, 2, 3, \ldots, n\}$ for $j = 1, 2, \dots, k_1 + k_2 + \cdots + k_m$. Now, for each cycle, we shall rewrite it such that the smallest element in the cycle appears first. For instance, if $a_j$ is the least element of $\{a_1, a_2, \ldots a_{k_1}\}$, then we rewrite the cycle $(a_1, a_2, a_3, \ldots, a_{k_1})$ by $(a_{j}, a_{j + 1}, \ldots ,a_{j - 2}, a_{j - 1})$. Similarly, we can use the same techniques for the cycles of $y$. We now arrange the cycles of $y$ as follows
        \begin{enumerate}
            \item From each cycle of $y$, collect the smallest element appearing in the cycle. Let $\mathcal{B}$ be the set consisting of all such elements.
            \item Since $\mathcal{B} \subseteq \mathbb{N}$, we have the smallest element in $\mathcal{B}$.
            \item Order the cycles of $y$ such that the first cycle corresponds to the smallest element in $\mathcal{B}$, the second cycle to the second smallest element and so on. This process terminates because $|\mathcal{B}| \leq n$.
        \end{enumerate}
        This ordering is well-defined, and henceforth, whenever we write the cycles, we assume they are written in accordance with this ordering.
        
        Let $x\in \mathfrak{S}_n$ with $\ord(x) = 4k$, for some $k \in \mathbb{N}$ and $\sigma:\mathfrak{S}_n \to \mathfrak{S}_n$ be the inner automorphism corresponding to $x$. Consider the set $S_{\sigma}$ corresponding to $\sigma$, given by
        \begin{align*}
            S_{\sigma} = \{y \in \mathfrak{S}_n \ | \ xyx^{-1} = y^{-1}\}
        \end{align*}

        We now address the broad category of inner automorphisms defined by a permutation $x$ where $\operatorname{ord}(x)$ is a multiple of 4. Understanding the complexity of characterizing $S_{\sigma}$. The core of our approach is an analysis of the cycle decomposition of $x$ relative to the constraints imposed on $y$ by the equation $\sigma_x(y) = y^{-1}$. We propose an algorithm for specific instances that leverages this structural information to construct an explicit injective map from the set of consisting of permutations $y$ to the set of standard involutions (elements of order 1 or 2) in $\mathfrak{S}_n$. However, this algorithm targets only a specific subset of permutations $x$ within this category where the structural constraints are tractable. The full resolution of the case where $\operatorname{ord}(x)$ is a multiple of 4 remains an open problem, and the subcases not covered by this algorithm are left for future work.
        \subsubsection{Mapping}
            Given $y \in S_\sigma$, we map $y$ to the elements of $S_I$ by
            
            \begin{enumerate}
                \item If $\ord(y) = 1, 2$, then map $y$ to itself.
                \item Otherwise, we have $\ord(y) \geq 3$. Suppose $x'$ is the element formed by extracting the cycles from $x$ such that $y'$ maps to its inverse under the conjugation map by $x'$, where $y'$ is the product of cycles from $y$ such that the cycle's order is greater than or equal to $3$. Now, we compute the value of $y*x'$ or $x'*y$ based on the structure of $x'$. Suppose the computed values are $y
                {''}$, then we follow the following steps.
                \begin{enumerate}
                    \item Combine, if $y{''}$ contains even number of elements as follows.
                    \begin{align*}
                        comb(y'') = &(a_1, a_2, a_3, \ldots, a_{k_1}, a_{k_1 + 1}, a_{k_1 + 2}, \ldots , a_{k_1 + k_2}, \ldots \\
                        &\ldots,a_{k_1 + \cdots + k_{m - 1} + 1}, a_{k_1 + 2}, \ldots , a_{k_1 + k_2 + k_3 + \cdots + k_m})
                    \end{align*}
                    If the number of elements appearing in $comb(y{''})$ is not an even number, we adjust it so that the even number of elements appearing in $comb(y{''})$ is even, and this procedure is the goal of this subsection.
                    \item Take two consecutive elements from $comb(y{''})$ and form a transposition. Repeat this process until we obtain $\frac{k_1 + k_2 + \cdots + k_m}{2}$
                    \begin{align*}
                        \Bar{y} = (a_1, a_2)(a_3, a_4) \cdots (a_{k_1 + k_2 + \cdots + k_m - 1}, a_{k_1 + k_2 + k_3 + \cdots + k_m})
                    \end{align*}
                    \item Since $\ord(\Bar{y}) = 2$, so we map $y \in S_{\sigma}$ to $\Bar{y} \in S_I$.
                \end{enumerate}
            \end{enumerate}
            This is well defined due to the ordering. Now, we are left to define the mapping to the case where the number of elements appearing in $comb((y{''})$ is not even, so now we examine when such a case arise.
            
            Let $y \in S_{\sigma}$ with $\ord(y) \geq 3$, then the possible cases of $y$ are as follows.
            \begin{enumerate}
                \item If $y$ is the product of disjoint cycles of distinct orders.
                \item If $y$ is the product of disjoint cycles of the same order.
                \item If $y$ is the product of disjoint cycles, some of which have the same order and some have different orders.
            \end{enumerate}
            \begin{enumerate}
                \item[Case (1).] If $y$ is the product of disjoint cycles of distinct orders, then the disjoint transpositions which map $y$ to its inverse appear in $x$. Suppose $x'$ is the product of disjoint cycles appearing in $x$ such that $y$ maps to its inverse under the conjugation map of $x'$. Since $x'$ is order $2$, so $y*x'$ is order 2 or 1. Suppose $y*x'$ is order 1 then $y = (x')^{-1}$ implying that $\ord(y) = 2$ which is a contradiction, so $y*x'$ is order 2.
                \item[Case (2).] If $y$ is a product of disjoint cycles, each with the same order, then it falls into one of the following four cases, where $y$ is of the form.
                \begin{align*}
                    y = (a_1, a_2, \ldots, a_j)(a_{j + 1}, a_{j + 2}, \ldots, a_{2j})\cdots(a_{(m - 1)j + 1},a_{(m - 1)j + 2}, \ldots, a_{mj}) \gamma
                \end{align*}
                In the following analysis, we restrict our attention to the specific case where the conjugation by $x$ maps a cycle of $y$ to the inverse of the next disjoint cycle. While there are multiple combinatorial ways to map these cycles to their inverses. For a cycle $(a_1, \ldots, a_j)$, $x$ conjugates this cycle to $(a_{2j}, \ldots, a_{j+1})$. Specifically, $x$ acts on the elements by $x(a_k) = a_{2j-k+1}$ for $1 \leq k \leq j$. Other combinatorial possibilities for the action of $x$ are outside the scope of this algorithm and remain a subject for future exploration.
                \begin{enumerate}
                    \item If $j$ is even and $m$ is even. Then $x'$ is of the following form.
                    \begin{align*}
                        &(a_1, a_{2j}, a_{2j + 1}, \ldots ,a_{(m - 2)j + 1}, a_{mj})(a_2, a_{2j - 1}, a_{2j + 2}, \ldots ,a_{(m - 2)j + 2}, a_{mj - 1})\\ 
                        &\ldots (a_j, a_{j + 1}, a_{3j}, \ldots ,a_{(m - 1)j}, a_{(m - 1)j + 1})
                    \end{align*}
                    and the value of $y*x'$ (first case) is
                    \begin{align*}
                        &(a_1, a_{j + 1}, \ldots ,a_{(m - 1)j + 1})(a_2, a_{2j}, \ldots ,a_{(m - 2)j + 2}, a_{mj})\\
                        &\ldots (a_j, a_{j + 2}, \ldots ,a_{(m - 1)j + 2})\gamma
                    \end{align*}
                    Here each cycle is an $m$-cycle and since $m$ is even we can proceed with the construction.
                    \item If $j$ is even and $m$ is odd. Then $x'$ is of the following form.
                    \begin{align*}
                        &(a_1, a_{2j}, a_{2j + 1}, \ldots ,a_{(m - 1)j + 1}, a_j, a_{j + 1}, a_{3j},\ldots,  a_{mj})\\
                        &(a_2, a_{2j - 1}, a_{2j + 2}, \ldots ,a_{(m - 1)j + 2}, a_{j - 1}, a_{j + 2}, \ldots ,a_{mj - 1})\\
                        &\cdots(a_{\frac{j}{2}}, a_{\frac{3j}{2} + 1}, \ldots ,a_{(m - 1)j + \frac{j}{2}}, a_{\frac{j}{2} + 1}, a_{\frac{3j}{2}}, \ldots, a_{(m - 1)j + \frac{j}{2} + 1}).
                    \end{align*}
                    and the value of $y*x'$ (first case) is
                    \begin{align*}
                        &(a_1, a_{j + 1}, \ldots, a_{(m - 1)j + 1})(a_2, a_{2j}, a_{2j + 2}, \ldots ,a_{(m - 1)j + 2}, a_{j},\ldots, a_{mj})\\
                        &\cdots (a_{\frac{j}{2}}, a_{2j - \frac{j}{2} + 2}, a_{2j + \frac{j}{2}}, \ldots ,a_{(m - 1)j + \frac{j}{2}}, a_{j - \frac{j}{2} + 2}, a_{2j - \frac{j}{2}}, \ldots, a_{(m - \frac{1}{2})j + 2})\\
                        &(a_{\frac{j}{2} + 1}, a_{\frac{3j}{2} + 1}, a_{\frac{5j}{2} + 1},\ldots ,a_{(\frac{m-1}{2})j + 1})\gamma
                    \end{align*}
                    The first and last cycles of $y * x'$ are $m$-cycles, while all the remaining cycles are $2m$-cycles. Since the total number of elements appearing in $y * x'$ is even, we can proceed with the construction.
                    \item If $j$ is odd and $m$ is even. Then $x'$ is of the following form.
                    \begin{align*}
                        (a_1, a_{2j}, a_{2j + 1}, \ldots, a_{mj})(a_2, a_{2j - 1}, \ldots, a_{mj - 1})\cdots(a_j, a_{j + 1}, a_{3j}, \ldots, a_{(m - 1)j + 1})
                    \end{align*}
                    and the value of $y*x'$ (first case) is
                    \begin{align*}
                        &(a_1, a_{j + 1}, \ldots ,a_{(m - 1)j + 1})(a_2, a_{2j}, a_{2j + 2}, \ldots, a_{mj})\cdots (a_j , a_{j + 2}, \ldots, a_{(m - 1)j + 2})\gamma
                    \end{align*}
                    Each cycle here is an $m$-cycle and since $m$ is even, we can proceed with the construction.
                \end{enumerate}
            \end{enumerate}
        The following subsection helps to us show that the mapping between $S_{\sigma}$ and $S_I$ we described is one-one.
        \subsection{One-one}
    
        Given $y_1, y_2 \in S_{\sigma}$ such that $\Bar{y_1} =  \Bar{y_2}$, then suppose $y_1$ is of the form.
        \begin{align*}
            y_1 = &(a_1, a_2, a_3, \ldots, a_{k_1})(a_{k_1 + 1}, a_{k_1 + 2}, \ldots , a_{k_1 + k_2}) \cdots \\
            &(a_{k_1 + k_2+ \cdots + k_{m - 1} + 1}, a_{k_1 + k_2 + \cdots + k_{m - 1} + 2}, \ldots , a_{k_1 + k_2 + \cdots + k_m})\notag
        \end{align*}
        As seen before, we have the following cases.
         \begin{enumerate}
            \item[Case (1).] If $y$ is the product of disjoint cycles of distinct orders, then we have already seen that $y_1*x'$ is order 2. Suppose there exits $y_2$ such that $\Bar{y_2} = y_1*x'$ and the possible cases of $y_2$ are
            \begin{enumerate}
                \item[Case (1.a)] If $\ord(y_2) = 2$, then $\Bar{y_2} = y_2 = y_1*x'$. Since $x = x'\delta$ where $\delta$ is the product of remaining disjoint cycles of $x$ that are not appearing in $x'$. Since the elements appearing in $x'$ are the only possible elements that appear in $y_2$, $\delta$ and $y_2$ do not have elements that appear in the cycles of $\delta$ and also in $y_2$, so $\delta^{-1} y_2 = y_2\delta^{-1}$. Since $y_2 \in S_{\sigma}$, so $xy_2x^{-1} = x'\delta y_2 (x'\delta)^{-1} = x'y_2(x')^{-1} = y_2^{-1}$ and using $y_2 = y_1*x'$, we obtain $x'y_1x'(x')^{-1} = (y_1x')^{-1}$ which implies $y_1^{2} = e$. But $\ord(y_1) \geq 3$, which is a contradiction. Therefore, such a $y_2$ does not exist.
                \item[Case (1.b)] Suppose $\ord(y_2) \neq 2$. If $\bar{y}_2 = y_1 x'$, then right-multiplying by $(x')^{-1}$ recovers $y_1$. To illustrate this, consider the 5-cycle $y_1 = (a_1, a_2, a_3, a_4, a_5)$ and the involution $x' = (a_1, a_5)(a_2, a_4)$. Composing these permutations yields $y_1 x' = (a_2, a_5)(a_3, a_4)$. If we assume $\bar{y}_2 = (a_2, a_5)(a_3, a_4)$, then right multiplying by $(x')^{-1}$ yields $y_2 = y_1 = (a_1, a_2, a_3, a_4, a_5)$.
            \end{enumerate}
            \item[Case (2).] If $y$ is product of disjoint cycles of same order.
            \begin{align*}
                y_1 = (a_1, a_2, \ldots ,a_j)(a_{j + 1}, a_{j + 2}, \ldots, a_{2j})\ldots(a_{(m - 1)j + 1},a_{(m - 1)j + 2}, \ldots, a_{mj})\gamma
            \end{align*} where $\gamma$ is product of disjoint 2-cycles. Assuming $j\ \geq\ 3$ and $m \geq 2$
            \begin{enumerate}
                \item[Case (2.a)] If $j$ is even and $m$ is even then $y_1'$ is $(a_1, a_2, \ldots, a_j)(a_{j + 1}, a_{j + 2}, \ldots, a_{2j})\ldots$\\ $(a_{(m - 1)j + 1},a_{(m - 1)j + 2}, \ldots, a_{mj})$ we already seen the corresponding $x'$ and the value of $y_1*x'$. By construction of the algorithm, $y_1$ is mapped to $\Bar{y_1}$, given by
                \begin{align*}
                    &(a_1, a_{j + 1})(a_{2j + 1}, a_{3j + 1})\ldots (a_{(m - 2)j + 1}, a_{(m - 1)j + 1})\\
                    &(a_2, a_{2j})(a_{2j + 2}, a_{4j})\ldots (a_{(m - 2)j + 2}, a_{mj - 1})\\
                    & \ldots (a_j, a_{j + 2})(a_{3j}, a_{3j + 2})\ldots (a_{(m - 1)j, a_{(m - 1)j + 2}})
                \end{align*}
                Suppose $y_2$ is mapped to $\Bar{y_1}$.
                So, now taking the every starting transposition i.e., $(a_1, a_{j + 1}), (a_2, a_{2j}), (a_3, a_{2j - 2}), \ldots (a_j, a_{j + 2})$ which results $(a_{j + 1}, a_{j + 2}, \ldots, a_{2j})$ and since there should be cycle which maps to its inverse under conjugation of $x$ similarly proceeding it gives the $y_1'$ and $\gamma$ as usual resulting $y_1$. So in this case it is one-one.
                
                \item[Case (2.b)] If $j$ is even and $m$ is odd, we have already seen the values of $x'$ and $y_1'*x'$. Recall the value of $y_1'*x'$ is
                $$\left(\begin{array}{l}
                    (a_1, a_{j + 1}, \ldots ,a_{(m - 1)j + 1})\\
                    (a_2, a_{2j}, a_{2j + 2}, \ldots, a_{(m - 1)j}, a_{(m - 1)j + 2}, a_{j}, a_{j + 2}, \ldots, a_{(m - 2)j + 2}, a_{mj - i + 2})\\
                    \cdots (a_i, a_{2j - i + 2}, a_{2j + i}, \ldots ,a_{(m - 1)j - i + 2}, a_{(m - 1)j + i}, a_{j - i + 2}, a_{j + i}, \ldots ,a_{(m - 2)j + i}, a_{mj - i + 2})\\
                    \cdots (a_{\frac{j + 1}{2}}, a_{2j - \frac{j + 1}{2} + 2}, a_{2j + \frac{j + 1}{2}}, \ldots, a_{(m - 1)j - \frac{j + 1}{2} + 1}, a_{j - \frac{j + 1}{2} + 2}, a_{j + \frac{j + 1}{2}}, \ldots, a_{(m - 2)j + \frac{j + 1}{2}}, a_{mk - \frac{j + 1}{2} + 2})
                \end{array}\right)$$
                By construction of the algorithm, it maps to $\Bar{y_1}$ given by
                \begin{align*}
                    (a_1, a_{j + 1})(a_{2j + 1}, a_{3j + 1})\cdots(a_{(m - 1)j + 1}, a_2)(a_{2j}, a_{2j + 2})\cdots
                \end{align*}
                
                Suppose $y_2$ is maps to $\Bar{y_1}$, since we know $a_1 < a_{j + 1}, a_{2j + 1}< a_{3j + 1}, \ldots, a_{(m - 3)j + 1}< a_{(m - 2)j + 1}$ so we use this information of $\Bar{y_1}$ and extract the information of $y_2$ by using $(x')^{-1}$, we obtain $y_2$ should be of the following form.
                \begin{align*}
                    (a_{j +1}, \underbrace{\ldots}_{?}, a_{2j})(a_{2j + 1},\underbrace{\ldots}_{?}, a_{3j})\cdots (a_{(m - 1) j + 1}, \underbrace{\ldots}_{?}, a_{mj})(a_1, \underbrace{\ldots}_{?}, a_j)
                \end{align*}
                
                Now, the cycle $(a_{(m - 1)j + 1}, a_2)$ has two possibilities.
                \begin{enumerate}
                    \item If $a_{(m - 1)j + 1}< a_2$, then on using $(x')^{-1}$, we obtain $a_j \to a_2$ which contradicts the fact that $a_j \to a_1$. So it is not possible.
                    \item If $a_2 < a_{(m - 1)j + 1}$ then on using $(x')^{-1}$, we obtain $a_{2j - 1} \to a_{(m - 1)j + 1}$ which contradicts the fact that $a_{mj} \to a_{(m - 1)j + 1}$. So not possible.
                \end{enumerate}
                
                So, we expect the merging of odd and even cycles has happened here. Now moving on next $2$-cycle $(a_{2j}, a_{2j + 2})$, we do have two possibilities.
                \begin{enumerate}
                    \item If $a_{2j} < a_{2j + 2}$ then on using $(x')^{-1}$, we obtain
                    \begin{align*}
                        (a_{j + 1}, a_{j + 2}, \underbrace{\ldots}_{?}, a_{2j - 1},a _{2j})(a_{2j + 1}, a_{2j + 2}, \underbrace{\ldots}_{?}, a_{3j - 1}, a_{3j})\cdots\\(a_{(m - 1)j + 1}, a_{(m - 2)j + 2}, \underbrace{\ldots}_{?},a_{mj - 1}, a_{mj})(a_1, a_2, \underbrace{\ldots}_{?}, a_{j - 1}, a_j)
                    \end{align*}
                    \item If $a_{2j + 2} < a_{2j}$ then on using $(x')^{-1}$, we obtain
                    \begin{align*}
                        &(a_{j + 1}, a_{3j + 2}, \underbrace{\ldots}_{?}, a_{4j - 1},a _{2j})(a_{2j + 1}, a_{4j + 2}, \underbrace{\ldots}_{?}, a_{5j - 1}, a_{3j})\cdots(a_{(m - 2)j + 1}, a_{2}, \underbrace{\ldots}_{?}, a_{(m - 1)j})\\&(a_{(m - 1)j + 1}, a_{j + 2}, \underbrace{\ldots}_{?},a_{2j - 1}, a_{mj})(a_1, a_{2j + 2}, \underbrace{\ldots}_{?}, a_{3j - 1}, a_j)
                    \end{align*}
                    using this information, on computing $y_2*x'$ we get
                    \begin{align*}
                        &(a_1, a_{j + 1}, \ldots ,a_{(m - 1)j + 1})\\
                    &(a_2, a_{mj}, a_{2j + 2}, \ldots
                    \end{align*} which implies that $\Bar{y_2} = (a_1, a_{j + 1})(a_{2j + 1}, a_{3j + 1})\cdots(a_{(m - 1)j + 1}, a_2)(a_{mj}, a_{2j + 2})\cdots$ which is not equal to $\Bar{y_1}$ because $(a_{2j}, a_{2j + 2})$ appears in $\Bar{y_1}$ but $(a_{mj}, a_{2j + 2})$ appears in $\Bar{y_2}$. So this is not possible.
                \end{enumerate}
                Similarly, repeating this process we get $y_1 = y_2$.
                \item[Case (2.c)] If $j$ is odd and $m$ is even, we already seen the values of $x'$ and $y_1*x'$. By construction of the algorithm,
                $y_1$ is mapped to $\Bar{y_1}$, given by
                \begin{align*}
                    (a_1, a_{j + 1})(a_{2j + 1}, a_{3j + 1}) \cdots (a_{(m - 2)j + 1)},a_{(m - 1)j + 1})(a_2, a_{2j})\cdots\gamma
                \end{align*}
                Suppose $y_2$ is mapped to $\Bar{y_1}$.
                So, now taking the every starting transposition i.e $(a_1, a_{j + 1}), (a_2, a_{2j}), (a_3, a_{2j - 2}), \ldots (a_j, a_{j + 2})$ which results $(a_{j + 1}, a_{j + 2}, \ldots, a_{2j})$ and since there should be cycle which maps to its inverse under conjugation of $x$ similarly proceeding it gives the $y_1'$ and $\gamma$ as usual resulting $y_1$. So in this case it is one-one.
            \end{enumerate}
        \end{enumerate}

    \begin{theorem}\label{identity_greater_other}
        Let $\sigma$ be an an outer automorphism for $\mathfrak{S}_6$ or a inner automorphism of $\mathfrak{S}_n$ corresponding to a permutation whose order is not a multiple of 4. The number of involutions of identity automorphism is greater than or equal to the number of involutions of $\sigma$.
    \end{theorem}
    \begin{proof}
        Given $\sigma \in Aut(\mathfrak{S}_n)$. Consider the set $S_{\sigma}$ corresponding to $\sigma$ and given by,
        \begin{align*}
            S_{\sigma} =\{ y \in \mathfrak{S}_n \ | \ \sigma(y) = y^{-1}\}.
        \end{align*}
        And the following set corresponding to identity automorphism $I : \mathfrak{S}_n \rightarrow \mathfrak{S}_n$ defined as $I(x) = x $ for all $x \in \mathfrak{S}_n$.
        \begin{align*}
            S_I = \{y \in \mathfrak{S}_n \ | \ \text{ord(y)} = 2  \text{ or } \text{ord(y) } = 1\},
        \end{align*}
        \begin{enumerate}
            \item For $n = 1$, $S_{\sigma} = S_I = \{e\}$
            \item For $n = 2,\ $ $S_{\sigma} = S_I = \{e, (1,2)\}$. 
            \item For $n \geq 3$ and $\sigma(y) = xyx^{-1}$ for some $x \in \mathfrak{S}_n$. Consider the following map $\alpha: S_\sigma\rightarrow S_I$ defined by
            \begin{align*}
                \alpha(y) = 
                    \begin{cases}
                        xy, & \text{ if ord(x) = 2}\\
                        xyx^{-1}, & \text{ if ord(x)} \geq 3 \text{ \& ord(x) is odd}\\
                        x^{\text{$\frac{ord(x)}{2}$}}y & \text{ if ord(x)} \geq 3 \text{ \& ord(x)/2 is odd}
                    \end{cases}
            \end{align*}
            If we can prove $\alpha$ is well defined and one-one, then these two results will pave a path for the desired one i.e., $|S_{\sigma}| \leq |S_I|$, for every $\sigma$ an inner automorphism o f $\mathfrak{S}_n$ corresponding to a permutation whose order is not a multiple of 4.
            \begin{itemize}
                \item[Case (i)] If $\ord (x) = 2$, then $\alpha(y) = xy$. On observation, we have $\alpha(e) = x$ and $\alpha(x) = e$. We claim that $\alpha$ is one-one. Let $y_1, y_2 \in \mathfrak{S}_n$. Suppose
                \begin{align*}
                    \alpha(y_1) = \alpha(y_2)
                    \implies x_1y_1  = x_1y_2
                    \implies y_1 = y_2
                \end{align*}
                Hence $\alpha$ is one-one. Now, we claim that $\alpha(y)$ is order 1 or order 2 element. Let $y \in \mathfrak{S}_n$.
                \begin{align*}
                    \alpha(y)\alpha(y) &= (x \cdot y)\cdot(x \cdot y)\\
                    (x \cdot y)\cdot(x_1 \cdot y) &= (x \cdot y)\cdot(y^{-1} \cdot x) \ (\because\text{ $xyx^{-1} = y^{-1}$})\\
                    \implies &= x^2
                    \implies = e
                \end{align*}
                We have $\ord(\alpha(y)) = 2 \text{ or } 1$, whenever $ord(x) = 2$. Furthermore, the map $\alpha$ is surjective.
                \item[Case (ii)] If $\ord(x) \geq 3$ and $\ord(x)$ is odd then $ \alpha(y) = xyx^{-1}$. We claim that $\alpha$ is one-one. Let $y_1, y_2 \in \mathfrak{S}_n$. Suppose
                \begin{align*}
                    \alpha(y_1) = \alpha(y_2)
                    \implies xy_1x^{-1} = xy_2x^{-1}
                    \implies y_1 = y_2
                \end{align*}
                Hence $\alpha$ is one-one. Now, we claim that $\alpha(y)$ is order 1 or order 2 element. Let $y \in \mathfrak{S}_n$.
                \begin{align*}
                    (x \cdot y \cdot x^{-1})\cdot(x \cdot y \cdot x^{-1}) &= xy^2x^{-1}\\
                    \implies &= xex^{-1} (\text{From Lemma \ref{all_involutions_of_greater_order_3}})\\
                    \implies &= e
                \end{align*}
                We have $\ord(\alpha(y)) = 2 \text{ or } 1$, whenever $\ord(x) \geq 3$ and $\ord(x)$ is odd.
                \item[Case (iii)] if $\ord(x) \geq 3$ and $\ord(x)/2$ is odd then $\alpha(y) = x^{\frac{\ord(x)}{2}}y$. We claim that $\alpha$ is one-one. Let $y_1, y_2 \in \mathfrak{S}_n$. Suppose
                \begin{align*}
                    \alpha(y_1) = \alpha(y_2) \implies x^{\frac{\ord(x)}{2}}y_1 = x^{\frac{\ord(x)}{2}}y_2 \implies y_1 = y_2
                \end{align*}
                Hence $\alpha$ is one-one. Now, we claim that $\alpha(y)$ is order 1 order 2 element. Let $y \in \mathfrak{S}_n$.
                \begin{align*}
                    (x^{\frac{\ord(x)}{2}}y) \cdot (x^{\frac{\ord(x)}{2}}y) &= (x^{\frac{\ord(x)}{2}}y) \cdot e \cdot (x^{\frac{\ord(x)}{2}}y)\\
                    &=  (x^{\frac{\ord(x)}{2}}y) \cdot x^{-\frac{\ord(x)}{2}} \cdot x^{-\frac{\ord(x)}{2}} \cdot (x^{\frac{\ord(x)}{2}}y)\\
                    \implies &= (x^{\frac{\ord(x)}{2}} \cdot y \cdot x^{-\frac{\ord(x)}{2}} \cdot y)\  (\because \frac{\ord(x)}{2} \text{ is odd})\\
                    \implies &= y^{-1}\cdot y\\
                    \implies &= e
                \end{align*}
                We have $\ord(\alpha(y)) = 2 \text{ or } 1$, whenever $\ord(x) \geq 3$ and $\ord(x)/2$ is odd.
            \end{itemize}
            \item For $n \geq 3$, Lemma \ref{all_complete} restricts the existence of outer automorphisms to the case $n=6$. The explicit forms of these outer automorphisms are described in Lemma \ref{all_outer_automorphisms_forms_of_S6}. Using the computational algebra system Magma \cite{MR1484478} (refer to APPENDIX), we computed the number of involutions for all such automorphisms. The results indicate a maximum of 36 involutions, which is strictly less than 76, the number of involutions for the identity automorphism of $\mathfrak{S}_6$.
        Since $|S_{\sigma}| \leq |S_I|, \forall \ \sigma \in \text{Aut}(\mathfrak{S}_n)$ which concludes our argument.
        \end{enumerate}
    \end{proof}

    \begin{corollary}
        Let $x$ be an element of $\mathfrak{S}_n$ corresponds to an inner automorphism $\sigma$ of $\mathfrak{S}_n$. If $\ord(x) = 2$, then $|S_{\sigma}| = |S_I|$.
    \end{corollary}
    
    Our analysis concludes that the unexplored automorphisms of $\mathfrak{S}_n$ are the inner automorphisms defined by a permutation $x$ where $\operatorname{ord}(x) \geq 3$ and $\operatorname{ord}(x)/2$ is even. An explicit representation of this automorphisms yields complex structure with numerous subcases. While the algorithm proposed in $\S \ref{algorithm}$ resolves a subset of these instances, the remaining subcases in this category are left as open problems. However, combining the relations from Lemma \ref{frobenius_Schur}, with the generalization provided by Theorem \ref{half_main_result}, we obtain the following observations regarding the involutions of $\mathfrak{S}_n$.
    \begin{corollary}
         Let $n$ be a natural number with $n > 3$ and $\sigma : \mathfrak{S}_n\to\mathfrak{S}_n$ an automorphism then the number of involutions of $\sigma$ is not more than $\frac{n!}{2}$.
    \end{corollary}
    \begin{proof}
        The proof is divided into two parts
        \begin{itemize}
            \item[Case (i)] If $n$ is even. 
            From Remark \ref{number_of_second_order_terms}, the number of elements of order 2 is equal to \\$\Sigma_{k = 0}^{\frac{n}{2} - 1} \frac{n!}{2^{\frac{n - 2k}{2}}\left(\frac{n - 2k}{2}\right)!\left(2k\right)!} $. On expanding, we get
            \begin{align*}
                 &= \frac{n!}{2^{\frac{n}{2}}(\frac{n}{2})!} + \frac{n!}{2^{\frac{n - 2}{2}}(\frac{n - 2}{2})!2!} + \cdots + \frac{n!}{2(n - 2)!}\\
                &= \frac{n!}{2}\left(\frac{1}{2^{\frac{n - 2}{2}}(\frac{n}{2})!} + \frac{1}{2^{\frac{n - 2}{2}}(\frac{n - 2}{2})!2!} + \cdots + \frac{1}{2(n - 2)!}\right)\\
                & < \frac{n!}{2}\left(\frac{1}{2\cdot 2^{\frac{n - 2}{2}}} + \frac{1}{2\cdot 2^{\frac{n - 4}{2}}} + \cdots + \frac{1}{2\cdot 2}\right)\ (\because n \geq 4)\\
                &= \frac{n!}{2}\cdot\frac{1}{2}\left(\frac{1}{2^\frac{n - 2}{2}} + \frac{1}{2^\frac{n - 4}{2}} + \cdots + \frac{1}{2} + 1)\right)\\
                &= \frac{n!}{2}\cdot\frac{1}{2}\left(\frac{1 - (\frac{1}{2})^{\frac{n}{2} - 1}}{\frac{1}{2}})= \frac{n!}{2}(1 - (\frac{1}{2})^{\frac{n}{2} - 1}\right) < \frac{n!}{2} \\
                 \implies 1 + & \Sigma_{k = 0}^{\frac{n}{2} - 1} \frac{n!}{2^{\frac{n - 2k}{2}}\left(\frac{n - 2k}{2}\right)!\left(2k\right)!} \leq \frac{n!}{2}
            \end{align*}
            \item[Case (ii)] If $n$ is odd. 
            From Remark \ref{number_of_second_order_terms}, the number of elements of order 2 is equal to \\
            $\Sigma_{k = 0}^{\frac{n - 1}{2} - 1} \frac{n!}{2^{\frac{n - (2k + 1)}{2}}\left(\frac{n - (2k + 1)}{2}\right)!\left(2k + 1\right)!}$. On expansion, we obtain
            \begin{align*}
                &= \frac{n!}{2^{\frac{n - 1}{2}}(\frac{n - 1}{2})!} + \frac{n!}{2^{\frac{n - 3}{2}}(\frac{n - 3}{2})!3!} + \cdots + \frac{n!}{2(n - 2)!}\\
                &= \frac{n!}{2}\left(\frac{1}{2^{\frac{n - 3}{2}}(\frac{n - 3}{2})!} + \frac{1}{2^{\frac{n - 5}{2}}(\frac{n - 3}{2})!2!} + \cdots + \frac{1}{2(n - 2)!}\right)\\
                & < \frac{n!}{2}\left(\frac{1}{2\cdot 2^{\frac{n - 3}{2}}} + \frac{1}{2\cdot 2^{\frac{n - 5}{2}}} + \ldots + \frac{1}{2\cdot 2}\right)\ (\because n \geq 4)\\
                &= \frac{n!}{2}\cdot\frac{1}{2}\left(\frac{1}{2^\frac{n - 3}{2}} + \frac{1}{2^\frac{n - 5}{2}} + \ldots + \frac{1}{2} + 1\right)\\
                &= \frac{n!}{2}\cdot\frac{1}{2}\left(\frac{1 - (\frac{1}{2})^{\frac{n - 1}{2} - 1}}{\frac{1}{2}}\right)= \frac{n!}{2}\left(1 - \left(\frac{1}{2}\right)^{\frac{n - 1}{2} - 1}\right) < \frac{n!}{2}\\
                \implies 1 + &\Sigma_{k = 0}^{\frac{n - 1}{2} - 1} \frac{n!}{2^{\frac{n - (2k + 1)}{2}}\left(\frac{n - (2k + 1)}{2}\right)!\left(2k + 1\right)!} \leq \frac{n!}{2}
            \end{align*}
        \end{itemize}
        Cases (i, ii) gives us the desired results as required which finish our argument.
    \end{proof}

    \begin{corollary}
        $\mathfrak{S}_3$ is the only non-abelian symmetric group which has involutions more than $\frac{1}{2}$ number of elements in the group.
    \end{corollary}

    So far we have observed structural properties of symmetric group, now we proceed to understand its representations. The following theorem provides a formula for the dimensions of all irreducible representations of the symmetric group using the hook length formula. The obtained formula leads to several combinatorial problems involving partitions that sum to 'nice expressions.' While we have analyzed these patterns for $n \leq 11$, several questions remain. In particular, our observations lead to a formal conjecture [see Theorem \ref{open_question_for_formula}], while general case is left for future investigation.
     \begin{theorem}\label{sodoicrs}
        The sum of degrees of irreducible complex representations of $\mathfrak{S}_n$ is
        \begin{align*}
            1 + \Sigma_{i_1, i_2, i_3, \cdots i_n}\frac{n!}{\Pi_{l = 1}^{{\delta}_n}\Pi_{j = 1}^{i_{l}} [(i_l - j) + 1 + {\Large\Gamma}_l^j]}
        \end{align*} where $i_1 + i_2 + i_3 + \cdots + i_n = n,\ n > i_1 \geq i_2 \geq \cdots \geq i_n \geq 0; \quad {\Large\delta}_n = \Sigma_{k = 1}^{n}{\Large\Lambda}_k$ where
        
        \begin{align*}
            {\Large\Omega}_j(i_k) = \begin{cases}
                    1, \quad i_k \geq j\\
                        0 \quad \text{ otherwise }
            \end{cases}
        \end{align*}
        \begin{align*}
            {\Large\Lambda}_k = \begin{cases}
                    1,  \quad i_k > 0\\
                        0  \quad \text{ otherwise }
            \end{cases}
        \end{align*} 
        \begin{align*}
            {\Large\Gamma}_l^j = \begin{cases}
                    \Sigma_{k = l + 1}^{n}{\Large\Omega}_j(i_k),  \quad  l + 1 \leq n\\
                        \quad \quad \quad 0  \quad\qquad \text{ otherwise }
            \end{cases}
        \end{align*}
    \end{theorem}
    \begin{proof}
        From \cite{MR1824028}*{Theorem 3.10.2}, the dimension of irreducible representation corresponding to the partition $\lambda$ of $n$ is $\frac{n!}{\Pi_{i, j}h_{i, j}}$, where $h_{i,j}$ is hook length of $(i,j)^{th}$ box of young diagram corresponding to a partition of $n$. Consider the partitions $\lambda = (i_1, i_2, \ldots, i_n)$ of $n$ satisfying
        \begin{align*}
            i_1 + i_2 + i_3 + \cdots + i_n = n,\ n > i_1 \geq i_2 \geq \cdots \geq i_n \geq 0;
        \end{align*}
        On viewing the partition $\lambda = (i_1, i_2, \ldots, i_n)$ as young diagram we have 
        \[
            \begin{array}{c c c c c}
            \ydiagram{4} & \cdots &\ydiagram{3} & \cdots & i_1 \text{ boxes}\\
            \ydiagram{4} & \cdots &\ydiagram{3} & \cdots & i_2 \text{ boxes}\\
            \vdots & \ddots & \vdots & \vdots & \vdots\\
            \ydiagram{4} & \cdots & \ydiagram{3}& \cdots & i_n \text{ boxes}
            \end{array}
        \]
        Assuming the young diagram indexed by $(1, 1), \cdots, (1, i_1)$ for the first row and similarly for other rows. 
        Consider the function ${\Omega}_j(i_k)$ defined by
        \begin{align*}
            {\Large\Omega}_j(i_k) = \begin{cases}
                    1, \quad i_k \geq j\\
                        0, \quad \text{ otherwise }
            \end{cases}
        \end{align*}
        ${\Large\Omega}_j(i_k)$ specifically determines whether the $k^{th}$ row has a box at the $j^{th}$ column. As we check to the rows bottom to the $l^{th}$ row and we need to check from $i_{l + 1}$ to $i_n$ and also if we are dealing with $i_n$ row is not necessary to check the number of elements bottom to it. So we introduce the following ${\Gamma}_l^j$ function which counts the number of bottom boxes of a given index.
        \begin{align*}
            {\Large\Gamma}_l^j = \begin{cases}
                    \Sigma_{k = l + 1}^{n}{\Large\Omega}_j(i_k),  \quad  l + 1 \leq n\\
                        \quad \quad \quad 0,  \quad\qquad \text{ otherwise }
            \end{cases}
        \end{align*}
        So for a given index $(i, j)$ the hook length $h_{i, j} = (i_l - j) + \Gamma_l^j + 1$ and the product all the hook lengths of the young diagram can be found starting from first row and taking product of hook lengths of all columns then moving to other rows taking the product of each row and taking multiplication of these products. For going to a row, we need to check whether the row is empty or not i.e., by checking it $i_l$ is non-zero and we introduce the ${\Lambda}_k$ function which checks whether a given row is having boxes or not.
        \begin{align*}
            {\Large\Lambda}_k = \begin{cases}
                    1,  \quad i_k > 0\\
                        0  \quad \text{ otherwise }
            \end{cases}
        \end{align*}
        The product runs for all the indexes, where rows \textit{l} runs from $1$ to ${\delta}_n$ where ${\delta}_n$ is the number of non-zero rows i.e., ${\Large\delta}_n = \Sigma_{k = 1}^n{\Large\Lambda}_k$ and columns \textit{j}, run from $1$ to $i_l$ (i.e., number of boxes in each row). So the product of hook lengths for a given partition $i_1, i_2, \cdots , i_n$ is $\Pi_{l = 1}^{{\delta}_n}\Pi_{j = 1}^{i_{l}} [(i_l - j) + 1 + {\Large\Gamma}_l^j]$ and dimension of the corresponding representation to this partition is $\frac{n!}{\Pi_{l = 1}^{{\delta}_n}\Pi_{j = 1}^{i_{l}} [(i_l - j) + 1 + {\Large\Gamma}_l^j]}$ and now we add all such partitions and one partitions which does not include is $(n)$ whose dimension is $1$. So we add one to the sum. Hence we have the result that the sum of dimensions of irreducible complex representations of a symmetric group is $1 + \Sigma_{i_1, i_2, i_3, \ldots i_n}\frac{n!}{\Pi_{l = 1}^{{\delta}_n}\Pi_{j = 1}^{i_{l}} [(i_l - j) + 1 + {\Large\Gamma}_l^j]}$
    \end{proof}
    \begin{remark}
        The expression in Theorem \ref{sodoicrs} is equal to 1 plus the expression in Lemma \ref{number_of_second_order_terms} in respective conditions.
    \end{remark}
    Let $f^\lambda$ denote the dimension of the irreducible representation of the symmetric group $\mathfrak{S}_n$ indexed by the partition $\lambda \vdash n$. The total sum of dimensions is given by the formula in Theorem \ref{sodoicrs}.
    We have derived a summation formula (the "Nice Expression") which decomposes this sum into the trivial component plus $\lfloor n/2 \rfloor$ distinct integer terms
    
    $$
     1 + \sum_{k=0}^{\lfloor n/2 \rfloor - 1} A_k^{(n)}
    $$
    
    where $A_k^{(n)}$ is given by
    \begin{align*}
        = \begin{cases}
              \frac{n!}{2^{\frac{n - 2k}{2}}\left(\frac{n - 2k}{2}\right)!\left(2k\right)!} , & \text{ if $n$ is even,}\\
            \frac{n!}{2^{\frac{n - (2k + 1)}{2}}\left(\frac{n - (2k + 1)}{2}\right)!\left(2k + 1\right)!}, & \text{ if $n$ is odd}
        \end{cases}
    \end{align*}
    \begin{theorem}[Open Problem]\label{open_question_for_formula}
        There exists a surjective map $\Phi: \mathcal{P}(n) \setminus \{(n)\} \to \{0, 1, \dots, \lfloor n/2 \rfloor - 1\}$ such that the set of non-trivial Young diagrams corresponding to partition of $n$, partitions this set into disjoint fibers $\mathcal{S}_k = \Phi^{-1}(k)$, such that $A_k^{(n)} = \sum_{\lambda \in \mathcal{S}_k} f^\lambda$, for each $k$.
    \end{theorem}

    Observed Structural Properties: Based on computational verification until $n \le 11$, the fibers $\mathcal{S}_k$ exhibit the following stable properties:
    \begin{enumerate}
        \item For the maximal index $k_{max} = \lfloor n/2 \rfloor - 1$, the fiber consists exclusively of the primary hook partitions:
        \begin{equation*}
            \mathcal{S}_{k_{max}} = \{ (n-1, 1), (n-2, 1, 1) \}
        \end{equation*}
        \item For the index $k_{max} - 1 = \lfloor n/2 \rfloor - 2$, the fiber $\mathcal{S}_{k_{max} - 1} = \{(n -2 , 2), (n - 3, 2, 1), (n - 3, 1, 1, 1), (n - 4, 2, 1, 1)\}$
        \item For index $k_{\max} - 2 = \lfloor n/2 \rfloor - 3$, the fiber $\mathcal{S}_{k_{max} - 2} = \{(n - 3, 3), (n - 4, 3, 1), (n - 4, 2, 2), (n - 4, 1, 1, 1, 1), (n - 5, 3, 2), (n - 5, 2, 2, 1), (n -5, 2, 1, 1, 1), (n - 5, 1, 1, 1, 1, 1), (n -6, 3, 3), (n - 6, 2, 2, 2), (n - 6, 2, 2, 2), (n - 6, 2, 1, 1, 1)\}$
        \item For index $k_{\max} - i$, the fiber $\mathcal{S}_{k_{max} - i}$ contains the partition $( n - i - 1, i - 1)$
    \end{enumerate}
    \bibliography{ref}
\appendix
\section{Computational Verification via Magma}
\label{appendix:magma}

    In this appendix, we provide the Magma \cite{MR1484478} code used to verify Theorem \ref{identity_greater_other} and Theorem \ref{sodoicrs} for outer automorphisms of $\mathfrak{S}_6$. The following script calculates $T(G)$ and compares it against $|S_\alpha|$ for all outer automorphisms of $\mathfrak{S}_6$.
    Code executed
    \noindent
    \begin{verbatim}
        // Define the function to create and apply the homomorphism
        MapAndCheckIdentity := function()
            // Define the symmetric group S6
            S6 := SymmetricGroup(6);
        
            // Define the symmetric group S5 for permutations of the elements 1 to 5
            S5 := SymmetricGroup(5);
        
            // Generate all permutations of the symmetric group S5
            perms := [p : p in S5];
        
            // Convert each permutation to a list and store in a nested list
            nestedPerms := [Eltseq(p) : p in perms];
        
            // Initialize the list of elements mapped to their inverses
            mapped_elements := [];
            
            // Define the target group, also S6 in this case
            G := SymmetricGroup(6);
            // Define the permutations
            perm1 := S6!(1,2)(3,4)(5,6);
            perm2 := S6!(1,3)(2,5)(4,6);
            perm3 := S6!(1,4)(2,6)(3,5);
            perm4 := S6!(1,5)(2,4)(3,6);
            perm5 := S6!(1,6)(2,3)(4,5);
        
            // Add the permutations to a list
            perms := [perm1, perm2, perm3, perm4\\, perm5];
            // Iterate over each list of permutations in nestedPerms
            for i in [1..#nestedPerms] do
                // Convert each list of integers in nestedPerms[i] to a permutation in G
                perm_images := [G!(perms[nestedPerms[i][j]]) : j in [1..5]];
        
                // Define the homomorphism for this particular set of permutations
                hom := hom<S6 -> G | 
                    [S6!(1,2) -> perm_images[1],  
                     S6!(1,3) -> perm_images[2],  
                     S6!(1,4) -> perm_images[3],  
                     S6!(1,5) -> perm_images[4],  
                     S6!(1,6) -> perm_images[5]
                    ]>;
        
                // Function to check if an element is mapped to its inverse
                IsMappedToInverse := function(perm)
                    img := hom(perm);
                    return img eq Inverse(perm);
                end function;
        
                // Iterate over all elements in S6 and collect those mapped to their inverses
                for perm in S6 do
                    if IsMappedToInverse(perm) then
                        Append(~mapped_elements, <perm, hom(perm)>);
                    end if;
                end for;
                
                // Display the elements and their images under the homomorphism
          
        
                // Count the number of elements mapped to their inverses
                count := #mapped_elements;
                print "Number of elements mapped to their
                inverses under this homomorphism: ", count;
                
                // Clear the list for the next iteration
                mapped_elements := [];
            end for;
        end function;
        
        // Call the function to perform the mapping and checking
        MapAndCheckIdentity();
    \end{verbatim}
\end{document}